\documentclass{article}
\usepackage{amsmath,amssymb,amscd,amsthm,amsfonts}
\usepackage{graphicx} 
\usepackage{hyperref}
\usepackage{url}

\title{A note on realizing triples of slice, concordance and Seifert genera}
\author{Boning Wang}

\newtheorem{theorem}{Theorem}[section]
\newtheorem{lemma}{Lemma}[section]
\newtheorem{proposition}{Proposition}[section]

\begin{document}
\maketitle
\begin{abstract}
We show that every triple of integers $(a,b,c)$ satisfying
$1\leq a\leq b\leq c$ is realized as the smooth slice genus,
smooth concordance genus, and Seifert genus of a knot,
respectively. This proves a conjecture of Kearney.
\end{abstract}

\section{Introduction}

For a knot $K\subset S^3$, let $g_3(K)$ and $g_4(K)$ denote its
Seifert genus and smooth slice genus, respectively. Its smooth concordance genus is defined by
\[
g_c(K)=\min\{g_3(J)\mid J\text{ is smoothly concordant to }K\}.
\]
These invariants satisfy
\[
0\leq g_4(K)\leq g_c(K)\leq g_3(K).
\]

If $K$ is smoothly slice, then it is smoothly concordant to the
unknot, and hence $g_4(K)=g_c(K)=0$.
For nonslice knots, it is therefore natural to ask whether the
inequalities above are the only restrictions on the three genera.
More precisely, given any triple of integers $(a,b,c)$ satisfying
$1\leq a\leq b\leq c$, does there exist a knot $K$ such that
\[
\bigl(g_4(K),g_c(K),g_3(K)\bigr)=(a,b,c)?
\]
Kearney~\cite[Conjecture~1]{Kearney2014} conjectured that every
such triple can be realized.

Early examples of Casson and Nakanishi showed that the slice genus
and concordance genus need not agree; see~\cite{Livingston2004}.
Livingston~\cite[Theorem~1.5]{Livingston2004}
established the special case $(a,b,c)=(1,N,N)$ for every integer
$N\geq1$, constructing algebraically slice knots $K$ with
\[
g_4(K)=1,\qquad g_c(K)=g_3(K)=N.
\]
Kearney~\cite{Kearney2013} determined the concordance genera of
533 of the 552 prime knots with 11 crossings.

Using knot Floer homology,
Hom~\cite{Hom2015ConcordanceGenus} constructed topologically slice
knots with smooth slice genus one and arbitrarily large smooth
concordance genus.
More broadly, knot Floer homology has been used to obtain various
lower bounds for the smooth concordance genus, both directly and
through bounds on the slice genus; see, for example,
~\cite{Hom2015ConcordanceGenus,
Hom2016EpsilonUpsilon,Wang2016GenusFiltration,
OSS2017ConcordanceHomomorphisms,Livingston2017Upsilon,
KimLivingston2018SecondaryUpsilon,FellerParkRay2019UpsilonSatellites,
DHST2021MoreHomomorphisms}.

In this paper, we prove Kearney's conjecture.

\begin{theorem}\label{thm:main}
For every triple of integers $(a,b,c)$ satisfying
$1\leq a\leq b\leq c$, there exists a knot $K\subset S^3$ such that
\[
\bigl(g_4(K),g_c(K),g_3(K)\bigr)=(a,b,c).
\]
\end{theorem}
\smallskip
\noindent\textbf{Declaration of AI Use.}
The construction was discovered by ChatGPT. AI tools were also used to
assist with language polishing.

\section{Preliminaries}

For any knot $K\subset S^3$, we normalize the Alexander polynomial by requiring
\(
\Delta_K(t)=\Delta_K(t^{-1}),\Delta_K(1)=1.
\)
We use the following consequence of
Stoimenow's realization theorem~\cite[Theorem~3.1]{Stoimenow2010}.

\begin{theorem}\label{Stoimenow}
Let $\delta(t)\in\mathbb{Z}[t,t^{-1}]$ satisfy
\[
\delta(t)=\delta(t^{-1}),\qquad \delta(1)=1,
\]
and let $n>0$ be the largest exponent of $t$ occurring in
$\delta(t)$. Then there exists a knot $K$ such that
\[
\Delta_K(t)=\delta(t),\qquad u(K)\leq 1,\qquad g_3(K)=n.
\]
\end{theorem}
For a knot $K\subset S^3$, let $D_K(t)=t^d\cdot\Delta_K(t)$ where $d$ is the largest exponent of \(t\) occurring in $\Delta_K(t)$.
Then $D_K(t)\in\mathbb{Z}[t]$ and $D_K(0)\neq0$.
We use the following special case of
Livingston's Alexander polynomial bound
~\cite[Theorem~2.2]{Livingston2004}.

\begin{theorem}\label{Livingston}
Suppose that every irreducible factor of $D_K(t)$
is self-reciprocal and occurs with multiplicity one. Then
\[
g_c(K)\geq \frac12\deg D_K(t).
\]
Here, a polynomial is called self-reciprocal if
\(p(t)=t^{\deg p}p(t^{-1})\).
\end{theorem}
\section{Building blocks}

For each positive integer \(n\), consider the normalized Laurent polynomial
\[\delta_n(t)=1-2{(2-t-t^{-1})}^n.\]
The polynomial \(\delta_n(t)\) satisfies $\delta_n(t^{-1})=\delta_n(t)$ and $\delta_n(1)=1$.
By Theorem~\ref{Stoimenow}, there exists a knot \(A_n\) with
\[\Delta_{A_n}(t)=\delta_n(t),\quad u(A_n)\leq1,\quad g_3(A_n)=n.\]

\begin{lemma}
    \(\lvert\sigma(A_n)\rvert=2\) and \(g_4(A_n)=1\). 
\end{lemma}
\begin{proof}
Since $\Delta_{A_n}(-1)=1-2\cdot 4^n<0$, we have
\[
{(-1)}^{\sigma(A_n)/2}
=\operatorname{sgn}\Delta_{A_n}(-1)=-1.
\]
Thus $\sigma(A_n)\equiv 2\pmod 4$, and hence
\[
2\leq |\sigma(A_n)|
\leq 2g_4(A_n)
\leq 2u(A_n)
\leq 2.
\]
Therefore $|\sigma(A_n)|=2$ and $g_4(A_n)=1$.
\end{proof}
Replacing $A_n$ by its mirror image if necessary,
we assume that $\sigma(A_n)=2$.

Set \(P_n(t):=t^n\cdot\delta_n(t)=t^n+2{(-1)}^{n+1}{(t-1)}^{2n}\).
\begin{lemma}
    The polynomial $P_n(t)$ is irreducible over $\mathbb{Q}$.
\end{lemma}
\begin{proof}
    Set \(\alpha=2^{-\frac{1}{n}}\in(0,1)\). Since \(x^n-2\) is irreducible, \([\mathbb{Q}(\alpha):\mathbb{Q}]=n\).
    Let \(\xi\) be a root of
    \[x^2-(2-\alpha)x+1=0.\]
    This quadratic is irreducible over \(\mathbb{Q}(\alpha)\), so \([\mathbb{Q}(\xi):\mathbb{Q}]=2n\).
    On the other hand, \(\delta_n(\xi)=1-2{(2-\xi-\xi^{-1})}^n=1-2\alpha^n=0\).
    Thus \(P_n(\xi)=0\) and \(P_n\) is irreducible.
\end{proof}

For each integer $m\geq3$, let $T_m$ be the twist knot
with a negative clasp and $2m-1$ half twists. When $m=1$, this is the left-handed trefoil.
\begin{lemma}
    A Seifert matrix of \(T_m\) is 
    \[
    V_m=
    \begin{pmatrix}
    1&1\\
    0&m
    \end{pmatrix}.
    \] 
    Moreover, \(\Delta_{T_m}(t)=m(t+t^{-1})-(2m-1)\). And $\sigma(T_m)/2=g_4(T_m)=g_3(T_m)=1$.
\end{lemma}
\begin{proof}
    Take the standard genus-one Seifert surface of \(T_m\), viewed as a disk with two bands, one of which has \(m\) full twists.
    Let \(\alpha\) and \(\beta\) be the standard basis curves associated with the two bands, chosen so that the \(\alpha\)-band has one positive full twist, and the \(\beta\)-band has $m$ positive full twists.
    The two bands meet in the usual single clasp.
    We have \(\operatorname{lk}(\alpha^+,\alpha)=1\), \(\operatorname{lk}(\alpha^+,\beta)=1\), \(\operatorname{lk}(\beta^+,\alpha)=0\), \(\operatorname{lk}(\beta^+,\beta)=m\).
    These linking numbers give the Seifert matrix \(V_m\) above.
    The Alexander polynomial is therefore
    \(\Delta_{T_m}(t)=\det(t^{\frac{1}{2}}V_m-t^{-\frac{1}{2}}V_m^T)=m(t+t^{-1})-(2m-1)\).
    Since the symmetrized matrix $V_m+V_m^T$ is positive definite, we have $\sigma(T_m)=2$. As \(g_3(T_m)=1\), we have \(g_4(T_m)=1\).
\end{proof}

Set $Q_m(t):=t\Delta_{T_m}(t)=mt^2-(2m-1)t+m$, which is irreducible when $m\geq3$.

Let $S=6_1$ be the stevedore knot. Its invariants are
\[
g_3(S)=1,\qquad g_4(S)=g_c(S)=0,\qquad
\sigma(S)=0;
\]
see~\cite{KnotAtlas}.

\section{Proof of Theorem~\ref{thm:main}}
Now fix $1\leq a\leq b\leq c$ and put
\[n=b-a+1,\qquad r=c-b.\]

Define
\[J_{a,b}=A_n\#(\#^{a+1}_{m=3}T_m)\]
and
\[K_{a,b,c}=J_{a,b}\#(\#^r S).\]

\begin{proposition}
    $g_4(K_{a,b,c})=a$, $g_c(K_{a,b,c})=b$ and $g_3(K_{a,b,c})=c$.
\end{proposition}
\begin{proof}
    By subadditivity of the slice genus under connected sum, we have
    \[g_4(K_{a,b,c})\leq g_4(A_n)+\sum_{m=3}^{a+1}g_4(T_m)+rg_4(S)=1+a-1=a.\]
    By additivity of the signature under connected sum, we also have
    \[\sigma(K_{a,b,c})=2+2(a-1)=2a.\]
    The signature bound gives the reverse inequality, so
    $g_4(K_{a,b,c})=a$.

    Since $S$ is slice, the knots $K_{a,b,c}$ and $J_{a,b}$
    are smoothly concordant. Hence
    \[
    g_c(K_{a,b,c})=g_c(J_{a,b}).
    \]
    For the upper bound, $g_c(J_{a,b})\leq g_3(J_{a,b})=b$.
    For the lower bound, we apply Theorem~\ref{Livingston} to $D_{J_{a,b}}(t)$.
    \[D_{J_{a,b}}(t)=P_n(t)\prod_{m=3}^{a+1} Q_m(t).\]
    The factors $P_n,Q_3,\ldots,Q_{a+1}$ are pairwise distinct
    irreducible self-reciprocal polynomials, each occurring
    with multiplicity one. In the case \(n=1\), we have \(P_1(t)=Q_2(t)\), which is still distinct from \(Q_3(t),\ldots,Q_{a+1}(t)\).
    Hence Theorem~\ref{Livingston} gives
    \[g_c(J_{a,b})\geq\frac{2n+2(a-1)}{2}=b.\]

    For the Seifert genus, we have
    \[g_3(K_{a,b,c})=g_3(A_n)+\sum_{m=3}^{a+1}g_3(T_m)+r\cdot g_3(S)=n+a-1+r=c.\]
\end{proof}

\bibliographystyle{alpha}
\bibliography{references}
\end{document}